\documentclass[12pt]{amsart}
\usepackage{amscd,amssymb,epsfig,mathtools}
\usepackage[usenames,dvipsnames]{color}
\calclayout

\numberwithin{equation}{section}

\newcommand{\R}{\mathbb{R}}

\newcommand{\tir}[1]{\ensuremath{\overline {#1}}} 
\newtheorem{thm}{Theorem}[section] 
 
\newtheorem{lemma}[thm]{Lemma} 
 
\newtheorem{prop}[thm]{Proposition} 
 
\newtheorem{defn}[thm]{Definition} 
 
\newtheorem{rem}[thm]{Remark}

\def\whsq{\vbox to 5.8pt 
{\offinterlineskip\hrule 
\hbox to 5.8pt{\vrule height 
5.1pt\hss\vrule height 5.1pt}\hrule}}

\def\<{\langle} 
\def\>{\rangle} 
\def\PP{{\mathop{{\rm I}\kern-.2em{\rm P}}\nolimits}} 
\def\FF{{\mathop{{\rm I}\kern-.2em{\rm F}}\nolimits}}   
\def\ZZ{{\mathop{{\rm I}\kern-.2em{\rm Z}}\nolimits}} 
 
\newlength{\sidemargin} 
\begin{document}
\title[]{On uniqueness of weak solutions to the second boundary value problem for generated prescribed Jacobian equations.
}


\author{Gerard Awanou}
\address{Department of Mathematics, Statistics, and Computer Science, M/C 249.
University of Illinois at Chicago, 
Chicago, IL 60607-7045, USA}
\email{awanou@uic.edu}  
\urladdr{http://www.math.uic.edu/\~{}awanou}

\author{Gantumur Tsogtgerel}
\address{Department of Mathematics and Statistics.
McGill University, 
Montreal, Quebec H3A 0B9, Canada}
\email{gantumur@math.mcgill.ca}  
\urladdr{http://www.math.mcgill.ca/gantumur/}

\begin{abstract}
We prove that two Aleksandrov solutions of a generated prescribed Jacobian equation have the same gradients at points where they are both differentiable and equal. For the optimal transportation case where two solutions can be translated to agree at a point without changing the $g$-subdifferential at that point, we recover the uniqueness up to a constant of solutions. For the general case, our result is a new proof with less regularity assumptions of a key theorem recently used to prove the uniqueness of solutions.

\end{abstract}

\maketitle

\section{Introduction}
Let $\Omega$ and $\Omega^*$ be two non empty bounded open and connected subsets of $\R^d$. We are interested in continuous functions $u$ on $\Omega$ which generate mappings $T_u: \Omega \to \Omega^*$ with a prescribed Jacobian, i.e.
\begin{equation} \label{generated-eq}
\det D T_u(x) = \psi(x,u(x),T_u(x)), \quad T_u(x)=T(x, u(x), D u(x)),
\end{equation}
where $\psi$ and $T$ are functions on $\Omega \times \R \times \Omega^*$ which take values in $\R$ and $\R^d$ respectively. We will assume that $\psi$ is separable in the sense that
$$
\psi(x,u,p) = \frac{f(x)}{R(T(x,u,p))},
$$
for positive functions $f \in L^1(\Omega)$ and $R \in L^1(\Omega^*)$. The second boundary value condition consists in requiring 
\begin{equation} \label{generated-eq2}
T_u(\Omega) = \Omega^*.
\end{equation} 
Problem \eqref{generated-eq}-\eqref{generated-eq2} was introduced by Trudinger \cite{trudinger2012local} motivated by problems in geometric optics 
as a generalization of the Monge-Amp\`ere equation of optimal transportation. In this paper, we prove under the stated assumptions that the gradients of weak solutions to generated prescribed Jacobian equations are equal at points where they are both differentiable and equal. In the optimal transportation case, by translation, we obtain an analytical description of the geometric argument that Aleksandrov solutions are unique up to a constant \cite[Chap 8-Theorem 2]{Pogorelov73}. 
Our result provides a more unified account of the uniqueness problem for generated prescribed Jacobian equations, as the optimal transportation case can be derived from our results without assuming that solutions are $C^{1,1}$. Uniqueness results were proven in \cite{rankin2020distinct} under the assumption that the solutions are in $C^{1,1}(\Omega)$.

A necessary condition for the existence of weak solutions in the sense of Aleksandrov of \eqref{generated-eq}-\eqref{generated-eq2} is the compatibility condition
\begin{equation} \label{generated-eq3}
\int_{\Omega} f(x) d x= \int_{\Omega^*} R(q) dq.
\end{equation} 
Weak solutions to $\det D^2 u(x)=f(x)$ and \eqref{generated-eq3}, in the case $R=1$, were called extremal solutions in \cite{urbas1984elliptic}. Our proof that gradients are equal at points where solutions are equal and differentiable, is analogous to arguments given in \cite{urbas1984elliptic} for extremal solutions. 

We organize the paper as follows. In the next section we review concepts on generated prescribed Jacobian equations pertinent to our results. In section \ref{uniqueness} we present our main result from which we derive uniqueness results. 

\section{Preliminaries}

We make 
structural assumptions following \cite{Jiang2018,rankin2020distinct}. All the assumptions listed below, (A1)--(A7) and (A1*), will be assumed to hold for the results of this paper. Recall that a domain of $\R^d$ is a non empty open and connected subset of $\R^d$.

Let $\Omega'$ be a bounded domain such that $\tir{\Omega} \subset \Omega'$ and let $\Gamma$ be a domain such that $\Gamma \subset  \Omega' \times \tir{\Omega^*} \times \R$ for which the projections
\begin{equation*} 
I(x,y) = \{ \, z \in \R, (x,y,z) \in \Gamma\, \},
\end{equation*}
are open intervals. 
We consider a bounded $C^4(\Gamma)$ function $g$ which will be referred to as the generating function. We refer to \cite[Section 4]{Jiang2014} for detailed examples of generating functions $g$. Let
$$
\mathcal{U} = \{ \, (x, g(x,y,z), g_x(x,y,z)), (x,y,z) \in  \Gamma \, \}.
$$
Note that $\mathcal{U} \subset \Omega \times \R \times \R^d$. We furthermore assume that
\begin{enumerate}
\item[(A1)] for each $(x,u,p) \in \mathcal{U}$, there exists a unique $(x,y,z) \in \Gamma$ such that
$$
g(x,y,z) = u, \quad g_x(x,y,z) = p,
$$
\item[(A2)] $g_z<0$
\end{enumerate}
Assumption (A1) allows to define the mapping $T: \mathcal{U} \to \R^d$ and a scalar function $Z: \mathcal{U} \to \R$  
such that
\begin{align*}
g(x,T(x,u,p),Z(x,u,p)) &= u \\
g_x(x,T(x,u,p),Z(x,u,p)) &=p.
\end{align*}
With Assumption (A2) one defines the dual generating function $h$ by
\begin{equation} \label{H-def}
g(x,y,h(x,y,u)) = u,
\end{equation}
for $(x,y,u) \in \Gamma^*=\{ \, (x,y,g(x,y,z), (x,y,z) \in \Gamma)\, \}$. By the implicit function theorem, $h$ is $C^4$. We define
$$
J(x,y)= g(x,y,.)(I(x,y)),
$$
and require that

(A1*) the mapping $Q=-g_y/g_z$ is one-to-one in $x$ for all $(x,y,z) \in \Gamma$.

Since $h_y=-g_y/g_z$, condition (A1*) is dual to (A1). We furthermore assume that 

(A3) the $d \times d$ matrix $E=g_{xy} - (g_z)^{-1}g_{xz} \otimes g_y$ is invertible on $\Gamma$.

The last assumption allows to write \eqref{generated-eq} as a Monge-Amp\`ere equation \cite{Jiang2018}
$$
\det [D^2 u - g_{xx}(.,T(.,u,Du),Z(.,u,D u))] = \det E(.,T(.,u,D u),Z(.,u,D u)) \psi(.,u,D u). 
$$
This follows from, c.f. \cite{rankin2020distinct},
$$
D T(.,u,Du) = E^{-1} [D^2 u - g_{xx}(.,T(.,u,Du),Z(.,u,D u))]. 
$$
The functions $x \mapsto g(x,.,.)$ play the role hyperplanes play as support functions in the theory of convex functions. 
\begin{defn} \label{def-g}
A function $u \in C^0(\Omega)$ is said to be $g$-convex on $\Omega$ if for each $x_0 \in \Omega$, there exists $(y_0,z_0) \in  \R^d \times \R$ such that $ (x_0,y_0,z_0) \in \Gamma$ and
\begin{align*}
u(x_0) & = g(x_0,y_0,z_0) \\
u(x) & \geq g(x,y_0,z_0) \, \forall x \in \Omega \text{ such that } (x,y_0,z_0) \in \Gamma.
\end{align*}
\end{defn}
We have added to the definition the conditions $ (x_0,y_0,z_0) \in \Gamma$ and $(x,y_0,z_0) \in \Gamma$ since $g$ is defined on $\Gamma$. 
The function $x \mapsto g(x,y_0,z_0)$ is called a $g$-affine function. We say that it is a $g$-support to the graph of $u$ at $x_0$. Let $|| . ||$ denote the Euclidean norm of $\R^d$. We have \cite[Proposition 3.4]{jeong2020h}

\begin{lemma} \label{semi-convex}
A $g$-convex function $u$ is semi-convex, i.e. $u + C ||x||^2$ is convex for some constant $C$. 
\end{lemma}

Let $u \in C^0(\Omega)$ be a $g$-convex function on $\Omega$. The $g$-subdifferential of $u$ at $x_0 \in \Omega$ is defined as the set-valued function
$$
\partial_g u(x_0) = \{ \, y \in \Omega^*, \exists z_0 \in  I(x_0,y) \ \text{such that} \ 
 g(x,y,z_0) \ \text{is a $g$-support to} \ u \ \text{at} \ x_0 \,
\, \}.
$$
For $E \subset \Omega$, we define $\partial_g u(E) = \cup_{x \in E} \partial_g u(x)$. 
\begin{lemma} \label{subd-affine}
For $(y_0,z_0) \in  \R^d \times \R$ and $x_0 \in \Omega$ such that $ (x_0,y_0,z_0) \in \Gamma$ with $z_0 \in I(x_0,y_0)$, assume that $L(x)=g(x,y_0,z_0)$ is a $g$-support to the graph of a $g$-convex function $u$ at $x_0$. Then $ \partial_g L (\Omega) = \{ \, y_0 \, \}$. 
\end{lemma}

\begin{proof}
Clearly $ \{ \, y_0 \, \} \subset \partial_g L(x_0) \subset \partial_g L (\Omega)$. Conversely, if $y_1 \in  \partial_g L (\Omega)$, $\exists \, x_1 \in \Omega$ and $z_1 \in I(x_1,y_1)$ with $(x_1,y_1,z_1) \in \Gamma$, such that 
$x \mapsto g(x,y_1,z_1)$ is a $g$-support to $L$ at $x_1$. Since $g \in C^4(\Gamma)$ and
$L(x) \geq g(x,y_1,z_1)$ for all $x \in \Omega$, we have $D L(x_1)=g_x(x_1,y_1,z_1)$. But $D L(x_1)=g_x(x_1,y_0,z_0)$ by definition of $L$.  Also $L(x_1)=g(x_1,y_0,z_0)=g(x_1,y_1,z_1)$. By Assumption (A1) we get $y_1=y_0$. This completes the proof.
\end{proof}

We will also need, c.f. for example \cite[Proposition 1]{Guillen2019},

\begin{prop} \label{g-sub-diff}
If a $g$-convex function $u$ is differentiable at $x_0 \in \Omega$, then $\partial_g u(x_0)$ has only one element $y_0$ determined by $g_x(x_0, y_0,z_0)=D u(x_0)$ with $z_0=h(x_0,y_0,u(x_0))$. 
\end{prop}

We now recall the notion of weak solution used in \cite{trudinger2012local}. The $g$-Monge-Amp\`ere measure is defined as the set function on Borel sets
$$
M[u](B) = \int_{\partial_g u(B)} R(p) d p.
$$
A weak solution of \eqref{second} in the sense of  Aleksandrov is a $C^0(\Omega)$ $g$-convex function $u$ such that
\begin{align}
\begin{split} \label{second}
M[u](B) & = \int_B f(x) d x \text{ for all Borel sets } B \subset \Omega \\
\partial_g u (\Omega) & = \Omega^*.
\end{split}
\end{align}
To relate $\partial_g u (x_0)$ to $T_u(x_0)$ for $x_0 \in \Omega$, we need the assumption

(A4) the matrix function $A(.,u,p) = g_{xx}(.,T(.,u,p),Z(.,u,p))$ satisfies 
$$
(D_{p_k p_l} A_{i j}) \zeta_i \zeta_j \eta_k \eta_l \geq 0,
$$
in $\mathcal{U}$ for all $\zeta, \eta \in \R^d$ such that $\zeta \cdot \eta = 0$. 

Define for $y_0 \in \R^d$
$$
I(\Omega,y_0) = \{ \, z \in \R, (x,y_0,z) \in \Gamma \text{ for all } x \in \Omega\, \}.
$$
We consider the mapping $Q$ defined on $\Gamma$ by
$$
Q(x,y,z) = - \frac{g_y}{g_z} (x,y,z).
$$
The domain $\Omega$ is said to be $g$-convex with respect to $y_0 \in \R^d$ and $z_{y_0} \in I(\Omega,y_0)$ if the image of $\Omega$ by the mapping $x \mapsto Q(x,y_0,z_{y_0})$ is convex in $\R^d$. The domain $\Omega$ is sub $g$-convex with respect to $y_0 \in \R^d$ and $z_{y_0} \in I(\Omega,y_0)$ if the convex hull of the image of $\Omega$ by the mapping $x \mapsto Q(x,y_0,z_{y_0})$ is contained in $Q(\Gamma)$. 

Recall that $\partial u (x_0)$ denotes the subdifferential of $u$ at $x_0$. We have by \cite[Lemma 2.2]{trudinger2020local}

\begin{lemma} \label{char-g-subd-diff}
Let $x_0 \in \Omega$ and assume that the domain $\Omega$ is sub $g$-convex with respect to all $y \in T(x_0,u(x_0),\partial u (x_0))$ and $z_y \in I(\Omega,y)$. Under the assumptions A1, A2, A1*, A3 and A4, if $u$ is a $C^0(\Omega)$ $g$-convex function, then 
$$
\partial_g u (x_0) = T(x_0,u(x_0),\partial u (x_0)).
$$
\end{lemma}

Finally, we need an additional condition to control the gradients of $g$-convex functions

(A5) there exists constants $m_0\geq -\infty$ and $K_0\geq 0$ such that for all $(x,y,z) \in \Gamma$, $(m_0,\infty) \subset J(x,y)$ and $|g_x(x,y,z)|\leq K_0$ if $g(x,y,z) \geq m_0$.

Assumption (A5) allows to prove existence of a solution to \eqref{second}.
As with \cite{jeong2020h}, c.f. also \cite{Guillen2017} for stronger assumptions, we will assume that there are continuous functions $a$ and $b$ defined on $\Omega' \times \tir{\Omega^*}$ such that

(A6) the interval $[a(x,y), b(x,y)] \subset J(x,y)$ for all $x \in \Omega'$ and $y \in \tir{\Omega^*}$

(A7) a solution $u$ of \eqref{second} satisfies $a(x,y) < u(x) < b(x,y)$ for all $x \in \Omega$ and $y \in \partial_g u(x)$.

Let us denote by $\mathcal{N}_r(A)$ the $r$-neighborhood of the set $A$. The following proposition \cite[Proposition 2.12]{jeong2020h} uses Assumption (A6). 

\begin{prop} \label{cont-subd}
Let $u$ be a $g$-convex function. For $x \in \Omega$ and $\epsilon>0$, there exists $\delta >0$ such that if $||z-x|| < \delta$ we have
$$
\partial_g u(z) \subset \mathcal{N}_{\epsilon} (\partial_g u(x)).
$$
\end{prop}


\section{Uniqueness of solutions to generated prescribed 
Jacobian equations}
\label{uniqueness}

We denote by $B_{\epsilon}(p)$ the Euclidean ball of center $p$ and radius $\epsilon$ and $|B_{\epsilon}(p)|$ denotes its Lebesgue measure. Let $u$ and $v$ be two solutions of \eqref{generated-eq}
and let $x_0 \in \Omega$. 

Assume that both $u$ and $v$ are differentiable at $x_0$ with $D u(x_0) \neq D v (x_0)$ and $u(x_0)=v(x_0)$. By Proposition \ref{g-sub-diff}, $\partial_g u(x_0)$ has only one element. The same holds for $\partial u(x_0)$. By Lemma \ref{char-g-subd-diff}, $T(x_0,u(x_0),D u (x_0))$ is the only element of $\partial_g u(x_0)$. 

Let $p' = D  v (x_0)$ and put $p=T(x_0,v(x_0),p')$. Now, let $U$ denote  
the set of points in $\Omega$ for which $u(x) < v(x)$. We have $x_0 \in \partial U$. We first show that $\partial_g v(U) \subset \partial_g u(U)$ with $\partial_g v(U) = \partial_g u(U)$ up to a set of measure 0. 
Then we show 
that there exists $\epsilon >0$ such that $B_{\epsilon}(p) \subset  \partial_g u(U)$. This implies in particular that $U$ is non-empty. Finally we show that there exists $\delta > 0$ such that 
$|B_{\delta}(x_0) \setminus U  |=0$ which is not possible since $D u(x_0) \neq D v (x_0)$. We conclude that $D v (x_0) = D u (x_0)$.

\begin{thm} \label{main}
Let $u, v \in C(\Omega)$ be two $g$-convex solutions of \eqref{second} with $f, R>0$. 
We also assume that $u, v > m_0$ on $\Omega$, where $m_0$ is given by Assumption (A5), that the domain $\Omega$ is sub $g$-convex with respect to all $y \in T(x_0,u(x_0),\partial u (x_0))$ and $z_y \in I(\Omega,y)$ and the assumptions  (A1)--(A7) and (A1*) hold. 
Assume that $u$ and $v$ are differentiable at $x_0$ 
for some $x_0 \in \Omega$ with $u(x_0)=v(x_0)$. Then $D u(x_0) = D v (x_0)$.
\end{thm}

\begin{proof}
Assume that $D u(x_0) \neq D v (x_0)$ and let $U$ denote the set of points $x$ in $\Omega$ for which $u(x) <v(x)$. We have $x_0 \in \partial U$. 

{\it Part 1} 
We prove that $\partial_gv(U) \subset \partial_g u(U)$ with $\partial_g v(U) = \partial_g u(U)$ up to a set of measure 0.

We first show that $\partial_gv(U) \subset \partial_g u(U)$. Let $x_1 \in U$ and $y_1 \in \partial_gv(x_1)$. 
Since by assumption $\partial_g v(\Omega) = \partial_g u(\Omega)$, there exists $x_2 \in \Omega$ such that $y_1 \in \partial_g u(x_2)$. Let $z_1 \in I(x_1,y_1)$ and $z_2 \in I(x_2,y_1)$ such that $v(x) \geq g(x,y_1,z_1)$ for all $x \in \Omega$ with equality at $x=x_1$ and
$u(x) \geq g(x,y_1,z_2)$ for all $x \in \Omega$ with equality at $x=x_2$. We claim that $z_1<z_2$. Otherwise $z_2\leq z_1$ and since $g_z<0$, we have
\begin{align*}
u(x_1) \geq g(x_1,y_1,z_2) \geq g(x_1,y_1,z_1) =v(x_1).
\end{align*}
By our assumption $u(x_1) < v(x_1)$ and thus $v(x_1) < v(x_1)$. A contradiction. Therefore $z_1 <z_2$ and for any $r \in \Omega$
\begin{align*}
v(r) \geq g(r,y_1,z_1) > g(r,y_1,z_2).
\end{align*}
In particular, for $r=x_2$, we obtain $v(x_2) > g(x_2,y_1,z_2)=u(x_2)$ i.e. 
$x_2 \in U$ and $y_1 \in \partial_g u(U) $. We conclude that  $\partial_gv(U) \subset \partial_g u(U)$. 

Since $M[u](U)=M[v](U)$, we have $\partial_g u(U) = \partial_gv(U)$ up to a set of measure 0.

{\it Part 2} Let $p' = D v (x_0)$ and put $p=T(x_0,v(x_0),p')$. 
As discussed above, by Proposition \ref{g-sub-diff} and Lemma \ref{char-g-subd-diff}, $p$ is the unique element of $\partial_g v(x_0)$.
We show that there exists $\epsilon >0$ such that $B_{\epsilon}(p) \subset  \partial_g u(U)$.

Since $\partial_gv (\Omega) = \partial_g u (\Omega)$, there exists $x_1 \in \Omega$ such that $p \in  \partial_g u(x_1)$. 
Let $z_0 \in I(x_0,p)$ and $z_1 \in I(x_1,p)$ such that 
$$
w(x) = g(x,p,z_0),
$$ is a $g$-support to the graph of $v$ at $x_0$ and 
$$
l(x)=g(x,p,z_1),
$$ is a $g$-support to the graph of $u$ at $x_1$. 

After choosing $\epsilon >0$, in particular so that  $B_{\epsilon}(p) \subset \partial_g u(\Omega)$, we will show that for 
$q \in B_{\epsilon}(p)$, one can find $z'_1 \in \R$ such that for $m(x)=g(x,q,z'_1)$ we have
$g(x_1,q,z'_1)=u(x_1)$ and $v \geq w > m$ on $\Omega$. We then show that $q \in \partial_g u(x_2), x_2 \in \Omega$ with $v(x_2) > m(x_2) \geq u(x_2)$, i.e. $x_2 \in U$. 

We have 
$$
g(x_0,p,z_0) =v(x_0) = u(x_0)  \geq g(x_0,p,z_1).
$$
Since $g_z <0$ we have $z_0 \leq z_1$. Moreover
$$
v(x_1) \geq  g(x_1,p,z_0)  \geq g(x_1,p,z_1) = u(x_1), 
$$
i.e. $x_1 \in \tir{U}$. 

Assume that $v(x_1)=u(x_1)$. Then, from the above inequalities $ g(x_1,p,z_0)  = g(x_1,p,z_1)$ and thus $z_0 = z_1$ as $g_z <0$. But this implies that $u(x) \geq g(x,p,z_0)$ for all $x \in \Omega$ with $g(x_0,p,z_0) =v(x_0) = u(x_0)$, i.e. $p \in \partial_g u(x_0) \cap  \partial_g v(x_0)$. 
By Proposition \ref{g-sub-diff} we have
\begin{equation*}
g_x(x_0,p,z_0) = D u (x_0) \text{ and } g_x(x_0,p,z_0) = D v (x_0).
\end{equation*}
Thus $D u (x_0)=D v (x_0) $ which contradicts our assumption and we conclude that $x_1 \in U, z_0 < z_1$ and $p \in  \partial_g u(U)$.

By definition, a domain is a connected open subset. Since $p \in \partial_g u(\Omega)=\Omega^*$ which is open by assumption, we can find $\epsilon'>0$ such that $B_{\epsilon'}(p) \subset \Omega^*$. Since $g(x_1,p,z_1)=u(x_1)> m_0$, by the continuity of $g$, there exists $\epsilon''>0$ such that $g(x_1,q,z_1)> m_0$ for $||q-p||<  \epsilon''$. 

Let $q \in \Omega^*$ such that $g(x_1,q,z_1)> m_0$. By Assumption (A5) $(m_0,\infty) \subset J(x_1,q)$. As $u(x_1) > m_0$, we can find, using the definition of $J(x_1,q)$, $z'_1 \in I(x_1,q)$ such that $g(x_1,q,z'_1) = u(x_1)$. Put 
$$
m(x) = g(x,q,z'_1).
$$ 
We have by \eqref{H-def}
$$
z'_1 = h(x_1,q,u(x_1)) \text{ and } z_1 = h(x_1,p,u(x_1)). 
$$
Recall that $z_0 < z_1$ and define
$$
k =  \min_{x \in \tir{\Omega}}  w(x) - l(x)= \min_{x \in \tir{\Omega}} g(x,p,z_0)-g(x,p,z_1).
$$
We emphasize that $k>0$. 
By the uniform continuity of $g$ and $h$ on $\tir{\Omega}$, there exists $\epsilon^{'''}>0$ such that for $||p-q||< \epsilon^{'''}$ we have for all $x \in \Omega$
\begin{align*}
|l(x) - m(x)| & = |g(x,p,z_1) - g(x,q,z'_1)| \\
& = |g(x,p,h(x_1,p,u(x_1))) - g(x,q,h(x_1,q,u(x_1)))| < k.
\end{align*}


Choose $\epsilon < \min\{ \, \epsilon', \epsilon'', \epsilon^{'''}\, \}$. 
 We have for all $x \in \Omega$ and $q \in B_{\epsilon}(p)$
\begin{equation} \label{hm}
w(x)-m(x) = (w(x)-l(x)) + (l(x)-m(x)) > k + (-k) =0.
\end{equation}
Since $B_{\epsilon}(p) \subset \Omega^* = \partial_g u(\Omega)$ we can find $x_2 \in \Omega$ such that $q \in \partial_g u(x_2)$. Let $z_2 \in I(x_2,q)$ such that $u(x) \geq g(x,q,z_2)$ for all $x \in \Omega$ with equality at $x=x_2$. We have 
$
g(x_1,q,z'_1) = u(x_1) \geq g(x_1,q,z_2)
$, which gives $z'_1 \leq z_2$ as $g_z<0$. We then have using \eqref{hm}
$$
v(x_2) \geq w(x_2) > m(x_2) = g(x_2,q,z'_1) \geq g(x_2,q,z_2)=u(x_2).
$$
We conclude that $x_2 \in U$.

{\it Part 3} 
We obtain a contradiction and conclude that $D v(x_0)=D u(x_0)$.

Since $B_{\epsilon}(p) \subset  \partial_g u(U)$ and $ \partial_g v(U) =  \partial_g u(U)$ up to a set of measure 0, we have
\begin{equation} \label{c1}
B_{\epsilon}(p) \subset  \partial_g v(U) \text{ a.e.}
\end{equation}
By Proposition \ref{cont-subd}, since $v$ is differentiable at $x_0$ and thus $\partial_g v (x_0) = \{ \, p \, \}$, there is $\delta>0$ such that 
\begin{equation} \label{c2}
\partial_gv(B_{\delta}(x_0) ) \subset  B_{\epsilon}(p). 
\end{equation}
Therefore by \eqref{c1} and \eqref{c2}
\begin{equation*} 
\partial_gv(B_{\delta}(x_0) \setminus U ) \subset  \partial_g v(U) \text{ a.e.} 
\end{equation*}
Now, the set of vectors which are contained in the $g$-subdifferential of two distinct points is contained in a set of measure 0, c.f. \cite[p. 1674]{trudinger2012local}. We therefore have
$$
| \partial_gv(B_{\delta}(x_0) \setminus U ) |=0.
$$
With $B=B_{\delta}(x_0) \setminus U$ we have $M[v](B) = \int_{\partial_g v(B)} R(p) dp=0=\int_B f(x) dx$ with $R>0$ on $\Omega^*$ and $f>0$ on $\Omega$, as $v$ solves \eqref{second}. We obtain $|B_{\delta}(x_0) \setminus U|=0$. 
But since $D v(x_0)\neq D u(x_0)$, for any $\rho>0$ such that $B_{\rho}(x_0) \subset \Omega$, we have $|B_{\rho}(x_0) \setminus U |>0$. A contradiction. 

The theorem is therefore proved as outlined at the beginning of this section. 

\end{proof}


\begin{rem}
Assumptions (A1)-(A7) and (A1*) are standard in the theory of generated prescribed Jacobian equations. 
\end{rem} 



In the remaining part of this section, we discuss some consequences of Theorem \ref{main}. 

\subsection{The optimal transport case}
We take $g(x,y,z)=c(x,y)-z$ where $c$ is a $C^4$ cost function such that the assumptions A1, A*1, A3, A4 and A7 hold. We have $g_z=-1$ so that Assumption A2 holds. Here $I(x,y)=\R$ and $J(x,y)=\R$. Assumption A5 then holds for $m_0=-\infty$ and we recall that $\Omega' \times \tir{\Omega^*}$ is bounded with $c \in C^4(\Gamma)$. Assumption A6 trivially holds in this case.

Let $u$ be a $g$-convex function. And put $v=u + \alpha$ for $\alpha \in \R$. Then for $x_0 \in \Omega$, $\partial_g u(x_0) = \partial_g v(x_0)$ and hence if $u$ solves \eqref{second}, then $u+\alpha$ is also a solution.

Let then $u$ and $v$ be two solutions of  \eqref{second}. If $u$ and $v$ are differentiable at $x_0 \in \Omega$, by adding a constant to $v$, we may assume that $u(x_0)=v(x_0)$.  As $g$-convex functions are semi-convex by Lemma \ref{semi-convex}, they are differentiable a.e. By Theorem \ref{main} we obtain $D u(x_0) = D v(x_0)$. We conclude that $Du =D v$ a.e. By Poincar\'e's inequality, $u-v$ is locally constant and hence a constant if $u$ and $v$ are continuous on $\Omega$.

In the general case, it cannot be guaranteed that $u+\alpha$ for a constant $\alpha$, is a solution of \eqref{second} when $u$ is also a solution. 

\subsection{Rankin's uniqueness results}

Theorem \ref{main} is proved in \cite{rankin2020distinct} with a completely distinct approach under the assumptions that $f, R>0$ and $C^1$. In addition the solutions $u$ and $v$ were assumed therein to be in $C^{1,1}(\Omega)$. 

It was then proved in \cite{rankin2020distinct} that under the above stated assumptions, solutions which intersect in $\Omega$ are the same. Furthermore, if the solutions are in $C^2(\tir{\Omega})$ and with a further convexity assumption on $\Omega^*$, solutions which intersect on $\partial \Omega$ were shown to be the same.

\section*{Acknowledgements}

Gantumur Tsogtgerel was supported by an NSERC Discovery Grant and by the fellowship grant P2021-4201 of the National University of Mongolia. Gerard Awanou is grateful  to Farhan Abedin, Jun Kitagawa and Yash Jhaveri  for suggestions that improved a preliminary version of the manuscript.

\end{document}